%% file: CVXSOS.tex
\documentclass[11pt]{article}

\usepackage[margin=1in]{geometry}

\pdfoutput=1
\usepackage{amsmath}
\usepackage{amsfonts}
\usepackage{amssymb}
\usepackage{graphicx}
\usepackage{amsthm}
\usepackage{mathdots}
\usepackage[usenames,dvipsnames]{xcolor}
\usepackage{tikz}
\usepackage{color}
\usepackage{hyperref}
\hypersetup{colorlinks=true,
    linkcolor=BrickRed,
    citecolor=OliveGreen,
    urlcolor=MidnightBlue}

\usepackage{booktabs}
\usepackage{algorithmicx}
\usepackage{algpseudocode}
\usepackage{longtable}

\usetikzlibrary{arrows}
\usetikzlibrary{patterns}
\usetikzlibrary{matrix}

\theoremstyle{plain}
\newtheorem{thm}{Theorem}[section]
\newtheorem{proposition}[thm]{Proposition}

\newtheorem{theorem}[thm]{Theorem}
\newtheorem{lemma}[thm]{Lemma}
\newtheorem{problem}[thm]{Problem}

\newtheorem{corollary}[thm]{Corollary}

\theoremstyle{definition}

\newcommand{\RR}{\mathbb{R}}
\newcommand{\CC}{\mathbb{C}}
\newcommand{\OO}{\mathbb{O}}
\newcommand{\HH}{\mathbb{H}}
\newcommand{\DD}{\mathbb{D}}

\newcommand{\cS}{\mathcal{S}}

\newcommand{\EE}{\mathbb{E}}

\newcommand{\tr}{\textup{tr}}

\newcommand{\CS}{\textup{cs}}

\newcommand{\sos}{\textup{sos}}

\newcommand{\Spin}{\textup{Spin}}

\newcommand{\End}{\textup{End}}

\newcommand{\gap}{\textsc{gap}}

\newcommand{\conj}[1]{\overline{#1}}

\title{A convex form that is not a sum of squares}
\author{James Saunderson\thanks{Department of Electrical and Computer Systems Engineering, Monash University, Melbourne VIC 3800, Australia. Email: \texttt{james.saunderson@monash.edu}}}

\begin{document}
\maketitle
\input{abstract}
\input{introduction}

\input{convexity}

\input{symred}

\input{notSOS}

\input{discussion}

\section*{Acknowledgments}
James Saunderson is the recipient of an Australian Research Council Discovery Early Career Researcher Award (project number DE210101056) funded by the Australian Government.

\bibliographystyle{alpha}
\bibliography{refs-cvxsos}

\end{document}

%% file: abstract.tex
\begin{abstract}
	Every convex homogeneous polynomial (or form) is nonnegative.
	Blekherman has shown that there exist convex forms that are not sums of
	squares via a nonconstructive argument.  We provide an explicit example
	of a convex form of degree four in 272 variables that is not a sum of
	squares. The form is related to the Cauchy-Schwarz inequality over the
	octonions. The proof uses symmetry reduction together with the fact
	(due to Blekherman) that forms of even degree, that are near-constant
	on the unit sphere, are convex. Using this same connection, we obtain
	improved bounds on the approximation quality achieved by the basic
	sum-of-squares relaxation for optimizing quaternary quartic forms on
	the sphere.
\end{abstract}

%% file: introduction.tex
\section{Introduction}
\label{sec:intro}

A multivariate polynomial $p$ of degree $2d$ in $n$ variables is
\emph{homogeneous} if it satisfies $p(\lambda x) = \lambda^{2d}p(x)$ for all
$x\in \RR^n$. Throughout, we use the term \emph{form} to refer to a homogeneous
polynomial.  A form is convex if and only if it satisfies the midpoint
convexity condition $\frac{1}{2}\left(p(x)+p(y)\right) \geq p\left(\frac{1}{2}(x+y)\right)$ for all $x,y\in \RR^n$~\cite{ahmadi2013complete}. 
Convex forms arise naturally, for instance, in the study of polynomial norms~\cite{ahmadi2019polynomial}.
Every convex form is nonnegative since
\[ p(x) = \textstyle{\frac{1}{2}\left(p(x)+p(-x)\right) \geq p\left(\frac{1}{2}(x-x)\right)} = p(0) = 0\quad\textup{for all $x\in \RR^n$}.\]
A natural question, posed by Parrilo in 2007, is whether every convex form is a sum of squares. 

Hilbert~\cite{hilbert1888darstellung} showed that every nonnegative form in $n$ variables of degree $2d$ is a sum of squares
if and only if either $n=2$ or $2d=2$ or $(n,2d)=(3,4)$. 
Since every convex form is nonnegative, in these cases convex forms 
must also be sum of squares. Going beyond this, recently El Khadir~\cite{el2020sum} has established that 
convex quaternary quartic forms are always sums of squares, despite there being nonnegative quaternary 
quartic forms that are not sums of squares. 

Parrilo's question was, however, resolved in the negative when 
Blekherman~\cite{blekherman2012chapter} showed that for
any fixed $2d\geq 4$ there is a sufficiently large $n$ such that there exist $n$-variate 
convex forms of degree $2d$ that are not sums of squares. In particular, he showed that the volume of
a particular compact section of the cone of $n$-variate convex forms is
strictly larger than the volume of the corresponding compact section of the
cone of $n$-variate sums of squares.  Blekherman's argument can be made
quantitative, but it only establishes the existence of convex quartic forms (for instance)
that are not sums of squares for $n\geq 27179089915$.

A stronger notion of convexity for a form is \emph{sum-of-squares (sos) convexity}~\cite{helton2010semidefinite}. 
A form is sos-convex if and only if $\frac{1}{2}\left(p(x)+p(y)\right)-p\left(\frac{1}{2}(x+y)\right)$ is a sum of squares as a polynomial in $x$ and 
$y$~\cite{ahmadi2013complete}. By restricting to $y=-x$, it follows that every sos-convex 
form is a sum of squares. Consequently, any example of a convex form that is not a sum of squares must also be 
an example of a convex form that is not sos-convex. However, all of the currently known examples 
of forms that are convex but not sos-convex are sums of squares~\cite[Section 6]{ahmadi2013complete}. 

The aim of this paper is to present what appears to be the first explicit
example of a convex form (of degree $4$ in $n=272=16\times 17$ variables) that is not
a sum of squares. 

\paragraph{Optimizing forms on the unit sphere}
Our example is also of interest in relation to the fundamental problem of minimizing a form of degree $2d$
over the unit sphere, i.e., computing
\[ p_{\min} := \min_{\|x\|^2=1} p(x).\]
When $2d\geq 4$, this is a very rich class of optimization problems. It includes NP-hard problems 
such as finding the size, $\alpha(G)$, of the largest stable set of a graph $G=(V,E)$, since the associated quartic forms
$p_G(x) = \sum_{i\in V}x_i^4 + 2\sum_{\{i,j\}\in E}x_i^2x_j^2$ satisfy $[p_G]_{\min}=1/\alpha(G)$~\cite{de2002approximation,motzkin1965maxima}.

In general, a lower bound on $p_{\min}$ is given by the quantity
\begin{equation}
	\label{eq:pminsos}
	p_{\min}^{\sos} := \max_{\gamma}\; \gamma\qquad\textup{subject to}\qquad p(x)-\gamma\|x\|^{2d}\;\textup{is a sum of squares}.
\end{equation}
This optimization problem can be reformulated as a semidefinite program (see, e.g.,~\cite{parrilo2012chapter}). For forms with 
a modest number of variables and degree, or forms with additional structure, $p_{\min}^{\sos}$ 
offers a computationally tractable global lower bound on $p_{\min}$.

Suppose that $p_{\max}$ is the maximum value of a form $p$ of degree $2d$ on the unit sphere. 
One way to measure the quality of $p_{\min}^{\sos}$ as a lower bound on $p_{\min}$ is via the quantity
\[ \gap(p):= \frac{p_{\max} - p_{\min}^{\sos}}{p_{\max} - p_{\min}} \geq 1.\]
For example, if $p$ is the Motzkin form $p(x,y,z) = x^2y^4+y^2x^4-3x^2y^2z^2+z^6$ which is nonnegative but not a sum of squares, 
then a numerical computation reveals that $\gap(p) \approx 1.0046$, remarkably close to one. For the quartic forms $p_G$ associated with the stable set problem, it is straightforward to show that $\gap(p_G)< 2$ for all graphs. 
While general upper bounds are known on $\gap(p)$~\cite{nie2012sum,fang2020sum}, 
little appears to be known about explicit forms for which this quantity is large. 
In this paper we will give an explicit example of a quartic form $p$ for which $\gap(p) > 2$.

We now describe the connection between forms for which $\gap(p)$ is large, and forms that are convex but not a sum of squares. In the course of establishing that there are convex forms that are not sums of squares, Blekherman showed that if a form of degree $2d$ is bounded between $1-\frac{1}{2d-1}$ and $1+\frac{1}{2d-1}$ on the unit sphere, then it is convex~\cite[Theorem 4.75]{blekherman2012chapter}. 
In Section~\ref{sec:convexity} 
we derive the following simple consequence of Blekherman's theorem.
\begin{proposition}
 	\label{prop:gap-suff}
 	Let $p$ be a form of degree $2d$ such that $p_{\min}\|x\|^{2d}\leq p(x)
	\leq p_{\max}\|x\|^{2d}$ for all $x\in \RR^n$. If $\gap(p) > d$ then $p(x)
 	-[dp_{\min}-(d-1)p_{\max}]\|x\|^{2d}$ is convex but not a sum of
 	squares.
\end{proposition}

\paragraph{The example}
The example of a quartic form with $\gap(\cdot)>2$ we focus on in this paper is one of a family of quartic 
forms $\CS_k^\OO$ in $n=16k$ variables, associated with the Cauchy-Schwarz inequality over the octonions.

Let $\DD$ denote a finite dimensional real normed division algebra. By
Hurwitz's theorem~\cite{hurwitz1898ueber}, $\DD$ is either $\RR$, $\CC$, the quaternions $\HH$,
or the octonions $\OO$. (The necessary background on the octonions is
summarized in Section~\ref{sec:oct-spin}.)
If $x\in \DD$, let $\conj{x}$ denote its conjugate and $|x| = (\conj{x}x)^{1/2}$ denote
its norm. If $x,y\in \DD^{k}$ then define $\|x\|^2 = \sum_{i=1}^{k}|x_i|^2$ and
$(x,y) = \sum_{i=1}^{k}\conj{x_i}y_i \in \DD$.  For all $k\geq 1$ we have that 
\[\CS_k^\DD(x,y):= \|x\|^2\|y\|^2 - |(x,y)|^2 \geq 0\quad\textup{for all $x,y\in \DD^k$},\]
giving a generalization of the Cauchy-Schwarz inequality. 
For a proof in the octonionic case, see, e.g.,~\cite{kramer1998octonion}. 
All the other cases follow by restricting to a suitable subspace. 

The Cauchy-Schwarz form over the reals or the complex numbers is well-known to be a sum of squares. 
One way to see this is to express the Cauchy-Schwarz form as a determinant and use the Cauchy-Binet formula, giving
\[ \|x\|^2\|y\|^2 - |(x,y)|^2 = \det\begin{bmatrix} \|x\|^2 & (x,y)\\\conj{(x,y)} & \|y\|^2\end{bmatrix} 
	= \sum_{1\leq i<j\leq k} \left|\det\begin{bmatrix} x_i & x_j\\y_i & y_j\end{bmatrix}\right|^2.\]
		While the $2\times 2$ Hermitian determinantal representation of the Cauchy-Schwarz form remains valid over 
		$\HH$ and $\OO$ (for a suitable
		notion of determinant), the Cauchy-Binet formula does not hold in these cases, and 
		the Cauchy-Schwarz forms over $\HH$ and $\OO$ are, in fact, not sums of 
		squares~\cite{ge2018isoparametric}.

In the quaternionic case, it turns out that $\gap(\CS_k^\HH) \leq 2$ for all $k$, and so that quaternionic 
Cauchy-Schwarz forms cannot provide the example we are looking for. 
On the other hand, for sufficiently large $k$ the octonionic Cauchy-Schwarz forms give rise to convex forms that are not
sums of squares.
\begin{theorem}
	\label{thm:main}
	If $k\geq 17$ then $\gap(\CS_k^{\OO}) > 2$ and $\CS_{k}^\OO(x,y) + (1/4)(\|x\|^2+\|y\|^2)^2$
	is a convex quartic form in $16k$ variables that is not a sum of squares.
\end{theorem}
The forms 
\[ q_k^\OO(x,y):= \CS_k^{\OO}(x,y) + (1/4)(\|x\|^2+\|y\|^2)^2\]
are convex for all $k\geq 1$, and are 
sums of squares exactly for $k\leq 16$ (see Theorem~\ref{thm:main-final}). As such,  
the smallest example, among this family, of a convex form that is not a sum of
squares occurs when $k=17$.  Overall, though, there is no reason at all to expect the example is
minimal in terms of the number of variables, among all convex forms that are not sums of squares. 

The main difficulty in establishing Theorem~\ref{thm:main} is establishing that $q_k^{\OO}$ is not a sum of squares
for $k\geq 17$. To do so, we exploit the symmetry properties of $q_k^{\OO}$. It turns out that $q_{k}^{\OO}$ is invariant under a certain action of the compact Lie group 
$\textup{Spin}(9)\times O(k)$. Explicitly decomposing the space of quadratic forms in $16k$ variables
into irreducible representations under this action reveals that the decomposition is multiplicity-free. 
This means that the problem of deciding whether $q_{k}^{\OO}$ is a sum of squares can be reduced to a small 
linear program by using the symmetry reduction framework of Gatermann and Parrilo~\cite{gatermann2004symmetry}. 

The paper is organized as follows. In Section~\ref{sec:convexity} we establish the 
connection between convex forms that are not sums of squares, and sum-of-squares
bounds on the minimum of a form on the unit sphere.
Using this simple connection we establish that the forms $q_{k}^\OO$
are convex for all $k$. In Section~\ref{sec:symred} we recall basic facts about
representation theory and symmetry reduction of semidefinite programs and sum of squares
feasibility problems. In Section~\ref{sec:notsos} we apply the general symmetry reduction framework to
the specific semidefinite feasibility problem of deciding whether 
$q_{k}^{\OO}$ is a sum of squares. We give an explicit sum-of-squares decomposition of
$q_{16}^\OO$ and give an explicit certificate that $q_{17}^{\OO}$ is not a sum
of squares.  In Section~\ref{sec:discussion}, we discuss some open questions
related to the results presented in the paper.

%% file: convexity.tex
\section{Convexity and optimization on the sphere}
\label{sec:convexity}

As part of the proof that there are convex forms that are not sums of squares, Blekherman established the 
following sufficient condition for a form of even degree to be convex. 
\begin{theorem}[{{Blekherman~\cite{blekherman2012chapter}}}]
\label{thm:blekherman}
If $p$ is a form of degree $2d$ in $n$ variables and 
\[ \|x\|^{2d}\left(1-\frac{1}{2d-1}\right) \leq p(x) 
\leq \|x\|^{2d}\left(1+\frac{1}{2d-1}\right)\quad\textup{for all $x\in \RR^n$}\]
then $p$ is convex.
\end{theorem} 
In order to find a convex form that is not a sum of squares, 
one approach is to find a form $p$ for which the bound $p_{\min}^{\sos}$ on its minimum on the sphere, defined in~\eqref{eq:pminsos}, 
is far from its true minimum on the sphere.
The following simple consequence of Theorem~\ref{thm:blekherman} makes this precise.
\begin{lemma}
\label{lem:cvx-sos}
Let $p$ be an form of degree $2d$ in $n$ variables such that $p_{\min}\|x\|^{2d} \leq p(x) \leq p_{\max}\|x\|^{2d}$ for all $x\in \RR^n$. 
Then 
	$p(x)-[dp_{\min}-(d-1)p_{\max}]\|x\|^{2d}$ is convex.
Furthermore
				if $p_{\min}^{\sos} < dp_{\min} - (d-1)p_{\max}$ then $p-[dp_{\min}-(d-1)p_{\max}]\|x\|^{2d}$ is 
not a sum of squares.
\end{lemma}
\begin{proof}
	If $p_{\min}\|x\|^{2d} \leq p(x) \leq p_{\max}\|x\|^{2d}$ for all $x\in \RR^n$, then a simple computation shows that 
	\begin{align*}
		\|x\|^{2d}\left(1-\frac{1}{2d-1}\right) &\leq \|x\|^{2d}\left(1-\frac{1}{2d-1}\frac{p_{\max}+p_{\min}}{p_{\max}-p_{\min}}\right) + \frac{2}{2d-1}\frac{1}{p_{\max}-p_{\min}}p(x)\\& \leq \|x\|^{2d}\left(1+\frac{1}{2d-1}\right)\quad\textup{for all $x\in \RR^n$}.
	\end{align*}
Applying Theorem~\ref{thm:blekherman}, and multiplying through by the positive quantity
 $\frac{1}{2}(2d-1)(p_{\max}-p_{\min})$, establishes that
	$p-[dp_{\min}-(d-1)p_{\max}]\|x\|^{2d}$ is convex.
The final part of the statement follows directly from the definition of $p_{\min}^{\sos}$.
\end{proof}
Proposition~\ref{prop:gap-suff}, connecting forms with large $\gap(\cdot)$ values and 
forms that are convex but not sums of squares, can be readily deduced from Lemma~\ref{lem:cvx-sos}.
\begin{proof}[{Proof of Proposition~\ref{prop:gap-suff}}]
	If $p$ is a form of degree $2d$ and $\gap(p)>d$ then $p_{\max} - p_{\min}^{\sos} > d(p_{\max} - p_{\min})$. 
	Therefore, $p_{\min}^{\sos} < dp_{\min} - (d-1)p_{\max}$. Lemma~\ref{lem:cvx-sos} then tells us that 
	$p-[dp_{\min}-(d-1)p_{\max}]\|x\|^{2d}$ is not a sum of squares.
\end{proof}

Recall the family of nonnegative quartic forms 
$\CS_k^{\OO}(x,y) = \|x\|^2\|y\|^2 - |\langle x,y\rangle|_{\OO}^2$, 
in $16k$ variables, associated with the Cauchy-Schwarz inequalty over the octonions. These satisfy 
\begin{equation}
	\label{eq:bd-cs}
(\|x\|^2+\|y\|^2)^2(1/4) \geq \|x\|^2\|y\|^2 \geq  c_k^\OO(x,y) \geq 0\quad\textup{for all $(x,y)\in \OO^{2k}$},
\end{equation}
where the first inequality follows from the AM-GM inequality. Both inequalities are tight. Applying Lemma~\ref{lem:cvx-sos}
with $p_{\min} = 0$ and $p_{\max}=1/4$ and $2d=4$ we can conclude that the forms
\[ q_k^\OO(x,y) = \CS_k^\OO(x,y) + (1/4)(\|x\|^2+\|y\|^2)^2\]
are convex, for all $k\geq 1$. 

\subsection{Worst-case bounds for optimization of quartic forms on the unit sphere}
\label{sec:worst-case}

In cases where convex forms are sums of squares (but nonnegative forms are not necessarily sums of squares), 
we can use Proposition~\ref{prop:gap-suff} to give new upper bounds on $\gap(p)$.
In particular, in~\cite{el2020sum}, El Khadir showed that every convex quaternary quartic form is a sum of squares. Combining this 
with the contrapositive of Proposition~\ref{prop:gap-suff} gives the following result.
\begin{corollary}
	If $p$ is a quaternary quartic form then $\gap(p) \leq 2$. 
\end{corollary}

We now compare this result with other upper bounds on $\gap(p)$ from~\cite{nie2012sum,fang2020sum}. 
We note that the comparison bounds hold in much greater generality, so it is unsurprising that improvements
are possible in the very special case considered here.
From Nie~\cite{nie2012sum} we have that for any quaternary quartic form, 
\[ \gap(p) = \frac{p_{\max} - p_{\min}^{\sos}}{p_{\max} - p_{\min}} \leq \frac{1}{\delta_4} \approx\frac{1}{0.0559}\approx 17.88\]
where $\delta_4$ is defined in~\cite[Equation 2.10]{nie2012sum} and its numerical value is given in~\cite[Table 1]{nie2012sum}. 
A special case of~\cite[Theorem 6]{fang2020sum} shows that for any polynomial of degree $2d$ in $n$ variables we have 
\begin{equation}
	\label{eq:hamzabound}
	\gap(p) = \frac{p_{\max} - p_{\min}^{\sos}}{p_{\max} - p_{\min}} \leq 1+(B_{2d}/2)\rho_{2d}(n,d).
\end{equation}
In the case $n=4$ and $2d=4$ one can compute numerically that $\rho_4(4,2) \approx 5.426$. The quantity 
$B_4$ is the supremum of $\|f_{4}\|_{\infty}/\|f\|_{\infty}$ where $f$ is a polynomial
of degree $4$ on the sphere, $f_4$ is the fourth harmonic component of $f$, and $\|f\|_{\infty}$ is the maximum absolute 
value of $f$ on the sphere. Restricting to the case of quartic forms in four variables, the argument in~\cite[Appendix B]{fang2020sum}
gives $B_4\leq 8$. This shows that the right hand side of~\eqref{eq:hamzabound} is at most $1+4(5.426) \approx 22.704$.
A simple lower bound on $B_4$ is $B_4\geq 1$. This shows that the right hand side of~\eqref{eq:hamzabound} is at least $3.713$.

%% file: symred.tex
\section{Background on representation theory and symmetry reduction}
\label{sec:symred}
In this section we summarize some basic notation, terminology, and facts about
representation theory of compact groups over the reals that will be required
for what follows.  We then briefly review symmetry reduction for semidefinite
programming feasibility problems in general, before specializing to feasibility
problems related to checking whether a form is a sum of squares.

\subsection{Preliminaries on representation theory}
\label{sec:repthy}

Let $G$ be a compact group. A \emph{representation} $(V,\rho_V)$ of $G$ over
$\RR$ is a real vector space $V$ together with a group homomorphism $\rho_V: G
\rightarrow GL_{\RR}(V)$ from $G$ to invertible $\RR$-linear maps on $V$. 
We often write $g\cdot x$ instead of $\rho_V(g)x$ when the homomorphism is clear
from the context, and refer to $V$, itself, as the representation.  
Throughout, we
assume we are working with finite-dimensional representations. 
Associated with a compact group $G$
is the \emph{Haar measure}, denoted $\mu_G$, which we normalize so that
$\mu_G(G)=1$. 

If $G$ is compact, we can equip any real representation $(V,\rho_V)$ of $G$
over $\RR$ with a $G$-invariant inner product, with the property that $\langle
g\cdot x,g\cdot y\rangle = \langle x,y\rangle$ for all $x,y\in V$ and $g\in
G$. With this choice, the representation is \emph{orthogonal} in the sense that 
$\rho_V(g)$ is an orthogonal transformation for all $g\in G$.

If $U$ is a representation of $G$ over $\RR$, then a subspace $V\subseteq U$ is
\emph{invariant} if $g\cdot v \in V$ for all $v\in V$ and $g\in G$. If the only
invariant subspaces of $U$ are $\{0\}$ and $U$, then we say that $U$ is
an \emph{irreducible} representation. 
If $V$ is an invariant subspace of $U$, then so is its orthogonal complement $V^\perp$
with respect to a $G$-invariant inner product on $U$. 

If $(U,\rho_U)$ and $(V,\rho_V)$ are two representations of a compact
group $G$ over $\RR$, let $\textup{Hom}_G(U,V)$ denote the vector space of
$\RR$-linear maps $A:U\rightarrow V$ such that $\rho_V(g)A = A\rho_U(g)$ for
all $g\in G$. Schur's lemma~\cite{schur1901klasse} tells us about the structure of these intertwining maps when $U$ and $V$ are 
irreducible.
\begin{lemma}[Schur's lemma]
	\label{lem:schur}
	If $U$ and $V$ are irreducible representations of $G$ over $\RR$ and $A\in \textup{Hom}_G(U,V)$ then 
	either $A=0$ or $A$ is invertible.
\end{lemma}
Two irreducible representations are \emph{equivalent} if there exists an
invertible linear map $A\in \textup{Hom}_{G}(U,V)$. Schur's lemma tells us that
the only intertwiner between inequivalent irreducible representations of $G$ is
the zero map. It also tells us about the \emph{endomorphsim algebra} of an
irreducible representation $U$, i.e., $\textup{End}_G(U) :=
\textup{Hom}_G(U,U)$ consisting of the self-intertwiners. In general, if $U$
is irreducible over $\RR$ then the endomorphism algebra is a division algebra
over $\RR$. As such, by the Frobenius theorem~\cite{frobenius1878lineare}, 
it is isomorphic either to $\RR$, $\CC$, or $\HH$. 


The following variation on Schur's lemma tells us that the only self-adjoint maps from an irreducible
representation to itself that commute with the group action are the multiples of identity.
Throughout, we use the notation $\cS^U$ for the self-adjoint linear maps from an inner product space 
$U$ to itself. 
\begin{proposition}
	\label{prop:schur-sa}
	Let $U$ be a representation of a compact group $G$ over $\RR$ 
	equipped with a $G$-invariant inner product. 
	If $U$ is irreducible then $\textup{End}_G(U)\cap \cS^U = \{\lambda I: \lambda\in \RR\}$. 
\end{proposition}
\begin{proof}
	Suppose that $U$ is irreducible and let $A\in \textup{End}_G(U)\cap \cS^U$. Since
	$A$ is self-adjoint, all of its eigenvalues are real. Let $\lambda$ denote any eigenvalue
	of $A$. Let $V$ be the nullspace of $A-\lambda I$. Then $V$ is a non-zero (since it contains an eigenvector) 
	invariant subspace of $U$. Since $U$ is irreducible, $V=U$ and so $A=\lambda I$.
\end{proof}
Since we can identify self-adjoint maps from $U$ to $U$ with bilinear forms on $U$, Proposition~\ref{prop:schur-sa}
tells us that an irreducible representation has a unique (up to scaling) $G$-invariant bilinear form. 

One way to study the representations of a group is via their characters.
If $(V,\rho_V)$ is a representation of $G$ over $\RR$, the associated \emph{character}
is the function $\chi_{V}:G\rightarrow \RR$ defined by $\chi_V(g) = \tr(\rho_V(g))$. 
In particular, the dimension of the space of intertwiners can be computed by integrating characters 
with respect to Haar measure. 
\begin{lemma}
	\label{lem:chi-ip}
	If $(U,\rho_U)$ and $(V,\rho_V)$ are two representations of a compact group $G$ over $\RR$ then 
\[ \langle \chi_U,\chi_V\rangle := \int_{g\in G} \chi_U(g)\chi_V(h)\;d\mu(G) = \dim_{\RR}(\textup{Hom}_G(U,V)).\]
\end{lemma}
If $(U,\rho_U)$ is a representation of $G$ over $\RR$ and $(V,\rho_V)$ is a representation of $H$ over $\RR$
then the \emph{(external) tensor product} $(U\otimes V,\rho_{U\otimes V})$ is a representation of $G\times H$ 
over $\RR$ defined by $(g,h)\cdot(u\otimes v) = (g\cdot u)\otimes (h\cdot v)$ and extending by bilinearity. 
The corresponding character is $\chi_{U\otimes V}(g,h) = \chi_{U}(g)\chi_V(h)$. Over $\CC$, the external 
tensor product of irreducible representations is always irreducible. When working over $\RR$, a little more 
care is required.
The following result will suffice for our purposes.
\begin{lemma}
	\label{lem:tp}
If $U$ and $V$ are irreducible representations of $G$ and $H$ respectively over
$\RR$, and $\dim_{\RR}\textup{End}_G(U) = \dim_{\RR}\textup{End}_H(V) = 1$, then 
$\dim_{\RR}\textup{End}_{G\times H}(U\otimes V) = 1$
and so $U\otimes V$ is an irreducible representation of $G\times H$ over $\RR$.
\end{lemma}
\begin{proof}
For the first part of the statement we compute
\begin{multline*}
\dim_{\RR}\textup{End}_{G\times H}(U\otimes V) = \langle \chi_{U\otimes V},\chi_{U\otimes V}\rangle = \\
\int_{G\times H} \chi_U(g)^2\chi_V(h)^2\;d\mu_G(g)d\mu_H(h) = \dim_{\RR}\textup{End}_G(U)\dim_{\RR}\textup{End}_H(V) = 1.\end{multline*}
For the second part of the statement, suppose that the endomorphism algebra $\textup{End}_{G\times H}(U\otimes V)$
is one-dimensional (and so contains only multiples of the identity) 
and let $W$ be an invariant subspace of $U\otimes V$.  Then the orthogonal projector onto $W$
is an element of the endomorphism algebra, and so must be a multiple of the identity. As such, the orthogonal
projector is either $0$ or $I$, implying that either $W = \{0\}$ or $W=U\otimes V$. 
\end{proof}

\subsection{Preliminaries on symmetry reduction}
\label{sec:symred-prelim}
In this section, we will briefly review symmetry reduction for semidefinite programming feasibility problems.

Let $V$ be a real inner product space and let $\cS^V\subseteq \textup{End}(V)$ 
denote the self-adjoint linear maps from $V$ to $V$. 
After choosing a basis the elements of $\cS^V$ can be identified with
symmetric $\dim(V)\times \dim(V)$ matrices. Let 
$\cS_+^V\subseteq \cS^V$ denote the cone of positive semidefinite elements of $\cS^V$, i.e., 
the elements $X\in \cS^V$ that satisfy $\langle v,Xv\rangle \geq 0$ for all $v\in V$.

A \emph{semidefinite feasibility problem} is a problem of the form
\begin{equation}
	\label{eq:sdp-feas}
	\textup{find}\;\; Y \in \cS_+^V\quad \textup{such that}\quad \mathcal{A}(Y)=b.
\end{equation}
where $\mathcal{A}:\cS^{V}\rightarrow W$ is a linear map and $b\in W$. 

\paragraph{Symmetry reduction for semidefinite feasibility problems}
Let $G$ be a compact group and let $V$ and $W$ be representations of $G$ 
over $\RR$ equipped with a $G$-invariant inner product. 
Then $\cS^V$ is also a representation of $G$ over $\RR$ via the action 
$Y \mapsto \rho_V(g) Y \rho_V(g)^\top$. If $\mathcal{A}\in \textup{Hom}_G(\cS^{V},W)$ and 
$g\cdot b = b$ for all $g\in G$ then we say that the semidefinite feasibility problem is \emph{$G$-invariant}.
This is because if $Y$ is feasible for~\eqref{eq:sdp-feas} then $\rho_V(g) Y\rho_V(g)^\top$ is also feasible 
for~\eqref{eq:sdp-feas} for all $g\in G$. 

If a semidefinite feasibility problem is $G$-invariant, then it is equivalent 
to the \emph{symmetry-reduced} feasibility problem
\begin{equation}
	\label{eq:sdp-feas-symred}
	\textup{find}\;\; Y\in \textup{End}_{G}(V) \cap \cS_+^V\quad\textup{such that}\quad \mathcal{A}(Y) = b.
\end{equation}
This is because if $Y\in \cS_+^V$ is feasible for~\eqref{eq:sdp-feas} then
\[ Y^G = \int_{g\in G} \rho_V(g) Y \rho_V(g)^\top\;d\mu_G(g) = 
\int_{g\in G} \rho_V(g) Y \rho_V(g^{-1})\;d\mu_G(g)\in \textup{End}_{G}(V)\cap \cS_+^V\]
and $\mathcal{A}(Y^G)=b$, so $Y^G$ is feasible for~\eqref{eq:sdp-feas-symred}. 

To proceed further, we need to consider how $V$ decomposes into irreducible representations. The general case 
(for finite groups) is discussed in~\cite{gatermann2004symmetry} in the language of invariant theory, and, for instance, in~\cite[Appendix 2]{raymond2018symmetric} in the 
language of representation theory.
For what follows, we need only focus on the very special case in which $V$ decomposes as a \emph{multiplicity-free} 
direct sum of inequivalent irreducible representations.
In this case, it follows from Lemma~\ref{lem:schur} and Proposition~\ref{prop:schur-sa} that 
$\textup{End}_G(V)\cap \cS^{V}$ is the span of the orthogonal projectors onto the inequivalent irreducible invariant subspaces
of $V$. Then the symmetry-reduced semidefinite feasibility problem~\eqref{eq:sdp-feas-symred} becomes 
a linear programming feasibility problem. Here, and throughout, we use the notation $\RR_+^k$ to denote the 
nonnegative orthant in $\RR^k$. 
\begin{proposition}
	\label{prop:sdp-symred}
	Suppose that $V$ is an orthogonal representation of a compact group $G$ over $\RR$, and 
	$V = \bigoplus_{i=1}^{k}V_i$ decomposes as a direct sum of inequivalent irreducible representations $V_i$ 
	of $G$ over $\RR$. If the semidefinite feasibility problem~\eqref{eq:sdp-feas} is $G$-invariant then it 
	is equivalent to
	\[ \textup{find}\;\; \lambda\in \RR_+^k\quad \textup{such that} \quad \sum_{i=1}^{k} \mathcal{A}(P_{V_i})\lambda_i = b\]
	where $P_{V_i}:V\rightarrow V$ is the orthogonal projector onto the subspace $V_i\subseteq V$. 
\end{proposition}
\paragraph{Symmetry reduction for sum-of-squares feasibility problems}
Next we specialize to the specific semidefinite feasibility problem
that arises when we want to check that a given $G$-invariant form is a sum of squares. 
Let $p\in \RR[x_1,\ldots,x_n]_{2d}$ be a form of degree $2d$ in $n$ variables. 
Let $V=\RR^{\binom{n+d-1}{d}}$, which 
we think of as the space of coefficients of real forms of degree $d$ in $n$ variables with respect to some fixed basis. With any $c\in V$
we use $c(x)$ to denote the associated form in $\RR[x_1,x_2,\ldots,x_n]_d$.
Define a linear map $\mathcal{A}:\cS^V\rightarrow W$ on rank one elements by 
\[ \mathcal{A}(cc^\top) = c(x)^2\]
and then extend to all of $\cS^V$ by linearity of $\mathcal{A}$ and by 
using the fact that an arbitrary element of $\cS^V$ can be written as 
$X = \sum_{i=1}^{\ell} c_ic_i^\top - \sum_{j=\ell+1}^{\ell'}c_{j}c_{j}^\top$. 

A form $p\in \RR[x_1,x_2,\ldots,x_n]_{2d}$ is a sum of squares if and only if 
the following problem is feasible:
\begin{equation}
	\label{eq:sos-feasibility}
	\textup{find}\;\; Y\in \cS_+^V\quad\textup{such that}\quad \mathcal{A}(Y) = p(x).
\end{equation}

Suppose, now, that $G$ is a compact group that acts by orthogonal transformations on $\RR^n$. 
The space of forms of degree $d$ becomes a representation of $G$ over $\RR$ via $(g\cdot p)(x) = p(gx)$ for all $g\in G$. 
This induces a representation on $V$, which is orthogonal once we equip $V$ with an invariant inner product.
If $p\in \RR[x_1,\ldots,x_n]_{2d}$ is fixed by the action of $G$, then the sum-of-squares 
feasibility problem~\eqref{eq:sos-feasibility}
is $G$-invariant. Suppose that $V$ decomposes in a multiplicity-free way into inequivalent 
irreducible representations as $V = \bigoplus_{i=1}^k V_i$. For $i=1,2,\ldots,k$, let $(c_{ij})_{j=1}^{\dim(V_i)}$ be an 
orthonormal basis for $V_i$. Then Proposition~\ref{prop:sdp-symred}
tells us that $p$ is a sum of squares if and only if there exist $\lambda_1,\ldots,\lambda_k\in \RR_+$ such that 
\begin{equation}
	\label{eq:sos-symred}
	p(x) = \sum_{i=1}^{k} \lambda_i s_i(x)\qquad\textup{where}\qquad s_i(x) = \sum_{j=1}^{\dim(V_i)} c_{ij}(x)^2\quad\textup{for $i=1,2,\ldots,k$}.
\end{equation}
This follows directly from the fact that the orthogonal projectors $P_{V_i}$ 
from Proposition~\ref{prop:sdp-symred} can be 
written explicitly as $P_{V_i} = \sum_{j=1}^{\dim(V_i)} c_{ij}c_{ij}^\top$.

%% file: notSOS.tex
\section{Sums of squares}
\label{sec:notsos}
In this section we apply the general results of Section~\ref{sec:symred} to the specific convex feasibility problem 
arising from checking that the forms $q_{k}^\OO$ are sums of squares. We begin in Section~\ref{sec:oct-spin} 
by recalling basic facts about the octonions, 
and understanding how the group $\Spin(9)$ acts on pairs of octonions. Then, in Section~\ref{sec:symqk}, we rewrite 
these forms in a way that reveals that they are invariant under the action of $\Spin(9)\times O(k)$ on 
$\OO^{k}\times \OO^{k}\cong \RR^{16\times k}$. In Section~\ref{sec:qf-decomp}, we decompose the space of quadratic forms 
on $\RR^{16\times k}$ into irreducibles under the action of $\Spin(9)\times O(k)$. This is the crucial representation-theoretic 
calculation required to describe the (polyhedral) cone of $\Spin(9)\times O(k)$-invariant quartic forms that are sums of squares. 
Finally, in Section~\ref{sec:sos-infeas}, we will show that $q_{16}^{\OO}$ is a sum of squares and $q_{17}^\OO$ is not a sum of squares, 
by studying the resulting linear programming feasibility problem in each case. 

\subsection{Background on the octonions and $\Spin(9)$}
\label{sec:oct-spin}

\paragraph{Octonions}
The octonions are the normed division algebra of dimension $8$. They can be explicitly constructed as the span of $8$ units 
$e_0,e_1,\ldots,e_7$ (with $e_0$ being the identity) 
subject to a particular non-associative multiplication (see, e.g.,~\cite[Equation 1.5]{tian2000matrix} 
for one possible multiplication table arising from a particular labeling of the units).

If $u = a_0e_0 + \sum_{i=1}^{7}a_ie_i$ then its \emph{conjugate} is $\conj{u} = a_0e_0 - \sum_{i=1}^{7}a_ie_i$ and its 
\emph{real part} is $\textup{Re}(u) = a_0$. Define a real inner product on $\OO$ 
via $\langle u,v\rangle = \textup{Re}(\conj{u}v)$ and let $|u|^2 = \langle u,u\rangle = \conj{u}u$
denote the corresponding norm. If $u = \sum_{i=0}^{7}a_ie_i$ then we denote by $[u]\in \RR^8$ the coordinate vector 
of $u$, i.e., $[u]_i = a_i$ for $i=0,1,\ldots,7$. Using this notation, $\langle u,v\rangle = [u]^\top[v]$. 
The squared norm can also be expressed as
\begin{equation}
	\label{eq:norm-alt}|u|^2 = \sum_{i=0}^{7} \langle u,e_i\rangle^2.
\end{equation}
Furthermore, left (resp.\ right) multiplication by $a$ is the adjoint of left (resp.\ right) multiplication by $\conj{a}$, i.e.,
\begin{equation}
\label{eq:LR-adj}\langle ax,y\rangle = \langle x,\conj{a}y\rangle\quad\textup{and}\quad \langle xa,y\rangle = \langle x,y\conj{a}\rangle\end{equation}
for all $x,y,a\in \OO$~\cite[Equation 2]{kramer1998octonion}. 
Given $u\in \OO$, let $R_u\in \End_{\RR}(\OO)$ denote right multiplication by $u = \sum_{i=0}^{7}a_ie_i$.
The matrix representing $R_u$, in the sense that $[R_u][v] = [vu]$ is~\cite{tian2000matrix}
\[ [R_u] = \begin{bmatrix} 
a_0 & -a_1 & -a_2 & -a_3 & -a_4 & -a_5 & -a_6 & -a_7\\
a_1 & a_0 & a_3 & -a_2 & a_5 & -a_4 & -a_7 & a_6\\ 
a_2 & -a_3 & a_0 & a_1 & a_6 & a_7 & -a_4 & -a_5\\
a_3 & a_2 & -a_1 & a_0 & a_7 & -a_6 & a_5 & -a_4\\
a_4 & -a_5 & -a_6 & -a_7 & a_0 & a_1 & a_2 & a_3\\
a_5 & a_4 & -a_7 & a_6 & -a_1 & a_0 & -a_3 & a_2\\
a_6 & a_7 & a_4 & -a_5 & -a_2 & a_3 & a_0 & -a_1\\
a_7 & -a_6 & a_5 & a_4 & -a_3 & -a_2 & a_1 & a_0\end{bmatrix}\]
(where we have used the choice of multiplication table in~\cite{tian2000matrix}).
Observe that $[R_{e_0}] = I$ and that $[R_{\conj{u}}] = [R_u]^T$. Furthermore, $[R_u]$ is skew-symmetric whenever
$u$ is purely imaginary, i.e., whenever $\textup{Re}(u)=0$. 

\paragraph{A model for $\Spin(9)$}
Define the following collection of $16\times 16$ real matrices 
\[ S_i := \begin{bmatrix} 0 & R_{e_i}\\R_{\conj{e_i}} & 0\end{bmatrix}\quad\textup{for $i=0,1,\ldots,7$ and}\quad
	S_8 := \begin{bmatrix} I & 0\\0 & -I\end{bmatrix}.\]
		This collection of symmetric matrices form a \emph{Clifford system}~\cite{ferus1981cliffordalgebren}, 
		in the sense that 
\begin{equation}
	\label{eq:Srel}S_iS_j + S_jS_i = 2I \delta_{ij}\;\;\textup{for $0\leq i,j\leq 8$}.
\end{equation}
Let 
\begin{equation}
\label{eq:V1def}
V_1:=\textup{span}\{S_i\;:\; i=0,1,\ldots,8\}\subseteq \RR^{16\times 16}
\end{equation}
be the $9$-dimensional subspace spanned by the $S_i$. 
Let $G$ denote the subgroup of $SO(16)$ such that $gV_1 g^\top = V_1$. 
This group is isomorphic to $\Spin(9)$, the simply connected double-cover of $SO(9)$, and is generated by 
elements of the form
\[ \begin{bmatrix} v & R_u\\ R_{\conj{u}} & -v\end{bmatrix}\quad\textup{where $u\in \OO, v\in \RR$ and $|u|^2+v^2 =1$}\]
(see, for instance, Harvey~\cite[Lemma 14.77]{harvey1990spinors}). 
From this description of $\Spin(9)$ we see that 
$\OO\times \OO\cong \RR^{16}$ is a representation of $\Spin(9)$ where the action is via matrix-vector multiplication.
We also see that $\Spin(9)$ acts on $\RR^{16\times 16}$ via $Z\mapsto gZg^\top$. The subspace $V_1$ is an invariant 
subspace under this action. The homomorphism $\rho_{V_1}:\Spin(9)\rightarrow O(9)$ given by the representation 
has image $SO(9)$ and kernel $\{I,-I\}$, and so is the double covering map from $\Spin(9)$ to $SO(9)$.

\subsection{Symmetries of $q_k^\OO$}
\label{sec:symqk}
In this section we 
rewrite the forms $q_k^\OO$ in an alternative way that makes their invariance properties clear. Before doing so, 
we will identify a pair $(x,y)\in \OO^k\times \OO^k$ with a matrix $X\in \RR^{16\times k}$ of their coordinates, in such a way that 
\begin{equation}
	X = \begin{bmatrix} [x_1] & [x_2] & \cdots & [x_k] \\ [y_1] & [y_2] & \cdots & [y_k]\end{bmatrix}.\label{eq:id-xy}
\end{equation}
From now on, we will think of $q_k^\OO$ as a polynomial in $X\in \RR^{16\times k}$. The following result 
expresses $q_{k}^\OO$ in a way that makes its symmetries clear. Recall that if $U$ is a subspace of an inner product space $V$,
 we write $P_{U}:V\rightarrow V$ for the orthogonal projector onto $U$. Throughout, we will use the trace inner product on $\RR^{16\times k}$, i.e., $\tr(X^\top Y)$, which is invariant for the group actions we consider.  
\begin{lemma}
	\label{lem:qreform}
Under the identifications in~\eqref{eq:id-xy} we can write $q_{k}^\OO(X)$ as 
	\[ q_k^\OO(X) = \frac{1}{2}\tr(XX^\top)^2 - \frac{1}{4}\sum_{i=0}^{8}\tr(X^\top S_i X)^2 = \frac{1}{2}\tr(XX^\top )^2 - 4 \|P_{V_1}(X X^\top)\|_F^2.\]
\end{lemma}
\begin{proof}
	Under the identifications in~\eqref{eq:id-xy} it is straightforward to check that
	\begin{align*}
		\|x\|^2+\|y\|^2 & = \sum_{j=1}^{k}\|[x_j]\|^2+\|[y_j]\|^2 =  \tr(XX^\top)\quad\textup{and}\\
		\|x\|^2-\|y\|^2  & = \sum_{j=1}^{k}\|[x_j]\|^2-\|[y_j]\|^2 = \tr(X^\top S_8 X).
	\end{align*}
	Using~\eqref{eq:norm-alt} and~\eqref{eq:LR-adj} we see that 
	\[ |(x,y)|^2 = \sum_{i=0}^{7}\langle \sum_{j=1}^{k}\conj{x_j}y_j,e_i\rangle^2 = \sum_{i=0}^{7}\sum_{j=1}^{k} \langle y_j,R_{e_i}x_j\rangle^2 = \frac{1}{4}\sum_{i=0}^{7} \tr(X^\top S_i X)^2.\]
	Combining these observations gives the first equality in the statement of the lemma. 
	For the second, we note that $S_i/4$ for $i=0,1,\ldots,8$ form an orthonormal 
basis for $V_1$ with respect to the trace inner product. Then $\sum_{i=0}^{8}\tr(X^\top S_i X)^2 = 16\sum_{i=0}^{8} \tr((S_i/4)X X^\top)^2 = 16\|P_{V_1}(XX^\top)\|_F^2$. 
\end{proof}
Recall that $\Spin(9)$ acts on $\RR^{16}$ as described at the end of Section~\ref{sec:oct-spin}. The orthogonal group $O(k)$ acts on 
$\RR^k$ via the defining representation. As such $\Spin(9)\times O(k)$ acts on $\RR^{16\times k}$, inducing an action on 
forms on $\RR^{16\times k}$ which fixes $q_k^\OO$.
\begin{lemma}
If $k\geq 2$ then $q_k^\OO(gXh^\top) = q_k^\OO(X)$ for all $g\in \Spin(9)$ and all $h\in O(k)$.
\end{lemma}
\begin{proof}
	Clearly $\tr(gXh^\top hX^\top g^\top)^2 = \tr(XX^\top)^2$ for all $g\in \Spin(9)$ and $h\in O(k)$ 
	since $\Spin(9)\subseteq SO(16)$.
	Furthermore, $\|P_{V_1}(gXh^\top h X^\top g^\top)\|_F^2 = \|g P_{V_1}(XX^\top) g^\top\|_F^2 = \|P_{V_1}(XX^\top)\|_F^2$
	since the orthogonal projector (with respect to a $G$-invariant inner product) onto a $G$-invariant subspace always 
	commutes with the action of the group. Since $q^{\OO}_k$ is in the span of these two invariant forms, it is also fixed by the action of $\Spin(9)\times O(k)$.
\end{proof} 

\subsection{Decomposing quadratic forms under $\Spin(9)\times O(k)$}
\label{sec:qf-decomp}
In order to perform symmetry reduction on the problem of deciding whether $q_k^\OO$ is a sum of squares, we need to 
decompose the space of quadratic forms in $\RR^{16\times k}\cong \RR^{16}\otimes \RR^{k}$ under the action of $\Spin(9)\times O(k)$. 
This is equivalent to decomposing the symmetric tensor square of $\RR^{16\times k}$ into irreducibles. 
We will do this by first decomposing the full tensor square $\RR^{16\times k}\otimes \RR^{16\times k}$
into irreducibles and then restricting to the subspace of appropriately symmetric tensors. To achieve this, it is convenient to use the 
identification $\RR^{16\times k}\otimes \RR^{16\times k} \cong \RR^{16\times 16}\otimes \RR^{k\times k}$ 
and separately decompose $\RR^{16\times 16}$ under the action of $\Spin(9)$, and $\RR^{k\times k}$ under the action of $O(k)$. 

Although the decompositions we will need are classical, it is not so straightforward to find concrete references to the 
results needed in a form that is broadly accessible, particularly as we are working with representations over $\RR$. 
As such, we will sketch character-theoretic proofs of the 
results we need. The approach we take makes use of the following beautiful formula relating integrals over the (special) orthogonal 
group to moments of Gaussian random variables. This formula is due to Diaconis and Shahshahani~\cite{diaconis1994eigenvalues} 
(for a smaller range of $n$) and extended to the range given below for $O(k)$ by Stoltz~\cite[Theorem 3.4]{stolz2005diaconis}, 
and for the range given below for $SO(k)$ by Pastur and Vasilchuk~\cite{pastur2004moments}.
\begin{theorem}
	\label{thm:ds-formula}
	Fix a positive integer $r$ and let $(a_1,a_2,\ldots,a_r)$ be an $r$-tuple of nonnegative integers. 
	Let $Z_1,Z_2,\ldots,Z_r$ be independent normally distributed random variables such that 
	\[Z_j\sim \mathcal{N}((1+(-1)^j)/2,j)\quad\textup{for $j=1,2,\ldots,r$}.\]
	If $\sum_{j=1}^{r}ja_j \leq 2k$ then 
	\[ \int_{h\in O(k)} \prod_{j=1}^{r}\tr(h^j)^{a_j}\;d\mu(h) = \prod_{j=1}^r \EE[Z_j^{a_j}].\]
	If $\sum_{j=1}^{r}ja_j \leq k-1$ then 
	\[ \int_{h\in SO(k)} \prod_{j=1}^{r}\tr(h^j)^{a_j}\;d\mu(h) = \prod_{j=1}^r \EE[Z_j^{a_j}].\]
\end{theorem}
\paragraph{Decomposing $\RR^{k\times k}$ under $O(k)$}
Recall that $O(k)$ acts on $\RR^{k}\otimes \RR^{k}\cong \RR^{k\times k}$ via $X \mapsto hXh^\top$. Let $U_0$ denote the span of the $k\times k$ identity matrix, $U_{1}$ the space of $k\times k$ traceless symmetric matrices, and $U_{-1}$ the space of $k\times k$ 
skew-symmetric matrices. These are clearly invariant subspaces that span $\RR^{k\times k}$. 
\begin{proposition}
	\label{prop:Ok-rep}
	If $k\geq 2$ then the subspaces $U_{j}$ for $j=-1,0,1$ are inequivalent irreducible representations for the action 
of $O(k)$ on $\RR^{k\times k}$ with $\dim_{\RR}(\textup{End}_{O(k)}(U_j))=1$.
\end{proposition}
\begin{proof}
	Let $\chi_0$, $\chi_{1}$ and $\chi_{-1}$ denote the character of $U_0$, $U_1$, and $U_{-1}$, respectively. 
	By Lemma~\ref{lem:chi-ip} it suffices to establish that $\|\chi_{0}\|^2 = \|\chi_{1}\|^2 = \|\chi_{-1}\|^2 = 1$. This is 
	a straightforward computation using Theorem~\ref{thm:ds-formula} and the fact that 
	$\chi_0(h)^2 = 1$, $\chi_1(h)^2 = [(\tr(h)^2 + \tr(h^2))/2-1]^2$ and $\chi_{-1}(h)^2 = [(\tr(h)^2 - \tr(h^2))/2]^2$. 
\end{proof}
\paragraph{Decomposing $\RR^{16\times 16}$ under $\Spin(9)$}
Recall that $\Spin(9)$ acts on $\RR^{16}\otimes \RR^{16}\cong \RR^{16\times 16}$ via $X\mapsto gXg^\top$. 
Given $J \subseteq \{0,1,\ldots,8\}$ with elements $J = \{j_1,j_2,\ldots,j_k\}$ 
satisfying $j_1<j_2<\cdots <j_k$, define 
\[S_{J} = S_{j_1}S_{j_2}\cdots S_{j_k}\]
if $J$ is non-empty, and $S_{\emptyset} = I$. The following result summarizes the key properties of these matrices that we will need. These follow from the relations~\eqref{eq:Srel} that define a Clifford system.
\begin{proposition}
	\label{prop:S-prop}
	The matrices $S_J$ for $|J|\in\{0,1,2,3,4\}$ form an orthonormal basis for $\RR^{16\times 16}$. Moreover, 
	$S_J$ is symmetric if $|J|\in \{0,1,4\}$ and $S_J$ is skew-symmetric if $j\in \{2,3\}$. 
\end{proposition}
\begin{proof}
	Clearly $S_{\emptyset} = I$ is symmetric and the matrices $S_j$ are symmetric by construction. 
	If $J = \{j_1,j_2,\ldots,j_k\}\subseteq \{0,1,\ldots,8\}$ with $j_1<j_2<\cdots < j_k$ then 
	\[ S_J^\top = S_{j_k}^\top S_{j_{k-1}}^\top \cdots S_{j_1}^\top = S_{j_k}S_{j_{k-1}} \cdots S_{j_1} = 
	(-1)^{\binom{k}{2}}S_{j_1}S_{j_2} \cdots S_{j_k} = (-1)^{\binom{k}{2}}S_J\]
	where we have reversed the order of the terms in the product in the last equality by
	applying~\eqref{eq:Srel} $\binom{k}{2}$ times. Since
	$\binom{k}{2}$ is odd if $k=2,3$, $S_{J}$ is skew-symmetric in these cases. 
	Since $\binom{k}{2}$ is even if $k=4$, $S_J$ is symmetric in this case.

 	We now show that the $S_J$ are mutually orthogonal for $|J|\in \{0,1,\ldots,4\}$. 
	By direct verification, we can check that $S_0S_1\cdots S_8 = -I$. Using this, together with~\eqref{eq:Srel}, 
	it follows that 
	\[ \tr(S_J^\top S_{J'}) = 0\quad\iff\quad \tr(S_{J \triangle J'}) = 0
	\quad\iff\quad \tr(S_{(J\triangle J')^c}) = 0\]
	where $J^c = \{0,1,\ldots,8\}\setminus J$ is the complement of $J$ and $J\triangle J'$ is the 
	symmetric difference of $J$ and $J'$.
	As such, to show that the $S_J$ are mutually orthogonal, it is enough to show that 
	$\tr(S_J)=0$ whenever $1\leq |J|\leq 4$. This clearly holds when $|J|=1$ because the $S_j$ have trace zero.
	This also holds for $k=2,3$ since in these cases $S_J$ is skew-symmetric.
	Finally, this holds for $k=4$ since if $J = \{j_1,j_2,j_3,j_4\}$ then $S_{j_1}$ is 
	symmetric and $S_{j_2}S_{j_3}S_{j_4}$ is skew-symmetric and so $\tr(S_J) = 0$. 

	That the $S_J$ for $|J|\in \{0,1,2,3,4\}$ span $\RR^{16\times 16}$ then follows from the fact that 
	$16^2 = 256 = 1+9+36+84+126 = \binom{9}{0} + \binom{9}{1} + \binom{9}{2}+ \binom{9}{3}+\binom{9}{4}$. 
\end{proof}
Define the subspaces $V_{k} \subseteq \RR^{16\times 16}$ for $k=0,1,2,3,4$ by
\[ V_{k} := \textup{span}\{S_{J}\;:\; J\subseteq \{0,1,\ldots,8\},\; |J|=k\}.\]
We have that $\dim(V_k) = \binom{9}{k}$ and that the $V_k$ are mutually orthogonal, spanning $\RR^{16\times 16}$. Moreover, 
$V_0,V_1$ and $V_4$ are subspaces of symmetric matrices and $V_2$ and $V_3$ are subspaces of skew-symmetric matrices.
\begin{proposition}
	\label{prop:spin9-rep}
The subspaces $V_k$ for $k=0,1,2,3,4$ are inequivalent irreducible representations for the action of $\textup{Spin}(9)$ on 
	$\RR^{16\times 16}$ with $\dim_{\RR}\textup{End}_{\Spin(9)}(V_k) = 1$.
\end{proposition}
\begin{proof}
	Recall that the representation $(V_1,\rho)$ of $\Spin(9)$ gives rise to a surjective homomorphism 
	$g\mapsto \rho(g)$ from $\Spin(9)$ to $SO(9)$ with kernel $\{I,-I\}$. The Haar measure on $\Spin(9)$
	pushes forward to the Haar measure on $SO(9)$ under this action.

	Let $\wedge^k V_1$ be the $k$th anti-symmetric tensor power of $V_1$.
	If $I = \{i_1,\ldots,i_k\}$ then the map 
	$S_I \mapsto S_{i_1}\wedge \cdots \wedge S_{i_k}$ gives an isomorphism 
	(of $\Spin(9)$ representations) between 
	the subspace $V_k$ and $\wedge^{k}V_1$ for $k=0,1,2,3,4$. 

	The $\wedge^k V_1$ are inequivalent irreducible 
	representations of $\Spin(9)$ with $\dim_{\RR}(\textup{End}_{\Spin(9)}(\wedge^k V_1)) = 1$ for all $k$. To see this,
	we first note that they are 
clearly invariant, and are inequivalent because they all have different dimensions. To see why they are irreducible, 
	we note that the character $\chi_k(h)$ of $\wedge^kV_1$ is the elementary symmetric polynomial in the eigenvalues of $\rho(h)$.
	Using the Newton identities, we can 
	express this character as a polynomial in the power sum symmetric functions $p_k(\lambda(\rho(h)) = \tr(\rho(h)^k)$. 
	Then by pushing forward the Haar measure on $\Spin(9)$ to the Haar measure on $SO(9)$ and applying 
	Theorem~\ref{thm:ds-formula} 
	we can check that $\|\chi_k\|^2=1$ for $k=0,1,2,3,4$. 
\end{proof}
We now combine the decomposition of $\RR^{16\times 16}$ into $\Spin(9)$ irreducibles, and the decomposition of $\RR^{k\times k}$ into $O(k)$ irreducibles to decompose $\RR^{16\times k}\otimes \RR^{16\times k}$ into irreducibles.
\begin{proposition}
	\label{prop:tensor-decomp}
	If $k\geq 2$, the space $\RR^{16\times k}\otimes \RR^{16\times k}\cong \RR^{16\times 16}\otimes \RR^{k\times k}$ decomposes under the action 
	$(g,h)\cdot(X\otimes Y) = gXh^\top \otimes gYh^\top$ of $\Spin(9)\times O(k)$ as
	\[ \RR^{16\times 16}\otimes \RR^{k\times k} = \bigoplus_{\substack{i\in \{0,1,2,3,4\}\\j\in \{-1,0,1\}}} 
	(V_i\otimes U_j)\]
	where each of the $V_i\otimes U_j$ are inequivalent irreducible representations of $\Spin(9)\times O(k)$. 
\end{proposition}
\begin{proof}
	This follows from Propositions~\ref{prop:Ok-rep} and~\ref{prop:spin9-rep} and Lemma~\ref{lem:tp}.
\end{proof}
We identify the coefficients of quadratic forms on $\RR^{16\times k}$ with elements of 
$\RR^{16\times 16}\otimes \RR^{k\times k}$ invariant under the action $(E\otimes F) \mapsto (E^\top \otimes F^\top)$. The identification associates $E\otimes F$ with the quadratic form $\tr(EXFX^\top)$. 
The irreducible representations
in the decomposition of quadratic forms under $\Spin(9)\times O(k)$ are exactly those irreducible representations 
from the decomposition of the tensor square $\RR^{16\times 16}\otimes \RR^{k\times k}$ in which either both 
factors are symmetric, or both factors are skew-symmetric. If we let 
\begin{equation}
	\label{eq:lambda-order}
	\Lambda = \{(0,0),(1,0),(4,0),(0,1),(1,1),(4,1),(2,-1),(3,-1)\}
\end{equation}
then $\Lambda$ indexes exactly these irreducible representations. The following 
result then follows immediately from Proposition~\ref{prop:tensor-decomp}.
\begin{proposition}
	\label{prop:coef-decomp}
	The space of coefficients of quadratic form on $\RR^{16\times k}$ 
	decomposes into inequivalent irreducible representations of $\Spin(9)\times O(k)$ over $\RR$ as 
	\[ W = \bigoplus_{(i,j)\in \Lambda} W_{ij}\quad\textup{where}\quad
	W_{ij} = \textup{span}\{E_i\otimes F_j\;:\; E_i\in V_i,\; F_j\in U_j\}\quad\textup{for $(i,j)\in \Lambda$.}\]
\end{proposition}

\subsection{Sum of squares feasibility}
\label{sec:sos-infeas}
We now apply the results of the previous two sections to perform symmetry reduction on the problem of deciding
whether $q_k^\OO$ is a sum of squares. We first give a more concrete description for the polynomials
that appear in the symmetry-reduced sum of squares decomposition. 

Let $(i,j)\in \Lambda$, let 
$E_{1},\ldots,E_{\dim(V_i)}$ be an orthonormal basis (with respect to the trace inner product) for $V_i$
and let $F_1,\ldots,F_{\dim(U_j)}$ be an orthonormal basis (with respect to the trace inner product) for $U_j$. 
Define
\[ s_{ij}(X) = \sum_{i'=1}^{\dim(V_i)}\sum_{j'=1}^{\dim(U_j)} \tr(X^\top E_{i'} X F_{j'})^2.\]
These quartic forms are exactly those that appear in~\eqref{eq:sos-symred}, after specializing to the
present context. The following is a restatement of~\eqref{eq:sos-symred} specialized to the case of 
quartic forms on $\RR^{16\times k}$ that are
invariant under the action of $\Spin(9)\times O(k)$.  
\begin{proposition}
	\label{prop:sos-char}
	Suppose that $f$ is a quartic form in $16k$ variables such that $f(gXh^\top) = f(X)$ for all $(g,h)\in \Spin(9)\times O(k)$. Then 
	$f$ is a sum of squares if and only if there exist $\lambda_{ij}\geq 0$ for $(i,j)\in \Lambda$
	such that 
	\[ f(X) = \sum_{(i,j)\in \Lambda} \lambda_{ij}s_{ij}(X) \quad\textup{for all $X\in \RR^{16\times k}$}.\]
\end{proposition}
This gives a linear programming feasibility problem to check whether $f$ is a sum of squares. To make 
the linear equality constraints relating the $\lambda_{ij}$ and $f$ more concrete, it is helpful to 
determine the dimension of the span of the $s_{ij}(X)$ for $(i,j)\in \Lambda$. 
\begin{proposition}
The space of  
$\Spin(9)\times O(k)$-invariant quartic forms is three-dimensional and is spanned by 
	$s_{00}(X)$, $s_{10}(X)$, and $s_{40}(X)$. 
\end{proposition}
\begin{proof}
	First we note that $s_{j0}(X) = \frac{1}{k}\|P_{V_j}(XX^\top)\|_F^2$ for $j=0,1,4$. 
	Recall that the space $\cS^{16}$ of $16\times 16$ symmetric matrices decomposes into inequivalent irreducible representations
	as $V_0\oplus V_1\oplus V_4$
	under the action of $\Spin(9)$ by $g\cdot Z \mapsto gZg^\top$. Therefore, the space of $\Spin(9)$
	invariant quadratic forms on $\cS^{16}$ is three dimensional, and is 
	spanned by $Z\mapsto \|P_{V_j}(Z)\|_F^2$ for $j=0,1,4$. 

	Since $U_0\oplus U_1 = \cS^{k}$, we have that 
	\[ s_{j0}(X) + s_{j1}(X) = \sum_{j'=1}^{\dim(V_j)}\|X^\top E_{j'}X\|_F^2 = \sum_{j'=1}^{\dim(V_j)}\tr(E_{j'}XX^\top E_{j'}XX^\top).\]
	The right hand side is a $\Spin(9)$-invariant quadratic form in $Z=XX^\top$ and so must be in the span 
	of $\|P_{V_j}(Z)\|_F^2$ for $j=0,1,4$. Therefore $s_{j1}(X)$ is in the span of $s_{00}(X)$, $s_{10}(X)$ and $s_{40}(X)$ for each $j=0,1,4$. 

	Consider, now, $s_{2,-1}(X)$ and $s_{3,-1}(X)$. Since $U_{-1}$ consists of all
	skew-symmetric matrices, 
	\[ s_{j,-1}(X) = \sum_{j'=1}^{\dim(V_j)}\|X^\top E_{j'}X\|_F^2 = \sum_{j'=1}^{\dim(V_j)}\tr(E_{j'}XX^\top E_{j'}^\top XX^\top) \quad\textup{for $j=2,3$}.\]
	Again, the right hand side is a $\Spin(9)$-invariant quadratic form in $Z = XX^\top$ and so must be in the span of 
	$\|P_{V_j}(Z)\|_F^2$ for $j=0,1,4$. Therefore, for $s_{j,-1}(X)$ is in the span of $s_{00}(X)$, $s_{10}(X)$, and $s_{40}(X)$ for $j=2,3$. 
\end{proof}
We are now ready to establish the main result of this section.
\begin{theorem}
	\label{thm:main-final}
	If $k\geq 17$ then $q_{k}^{\OO}$ is not a sum of squares. If $k\leq 16$ then $q_k^\OO$ is a sum of squares.
\end{theorem}
\begin{proof}
	If $k'\leq k$ then $q_{k'}^\OO$ is the restriction of $q_{k}^\OO$ to a
	subspace. Therefore, if $q_{17}^{\OO}$ is not a sum of squares, it
	follows that the same holds for $q_{k}^{\OO}$ with $k\geq 17$.
	Similarly, if $q_{16}^\OO$ is a sum of squares, the same holds for
	$q_{k}^\OO$ with $k\leq 16$. 

	Since the space of $\Spin(9)\times O(k)$-invariant quartic forms has dimension three, the 
	characterization of $\Spin(9)\times O(k)$-invariant sums of squares in Proposition~\ref{prop:sos-char}
	is equivalent to 
	\begin{equation}
		\label{eq:sos-eval}
	q_{17}^{\OO}(X_\ell) = \sum_{(i,j)\in \Lambda} \lambda_{ij}s_{ij}(X_\ell) \quad\textup{for $\ell=1,2,3$}
	\end{equation}
	for three points $X_1,X_2,X_3\in \RR^{16\times k}$ for which the $3\times 8$ coefficient matrix 
	has rank three. 
	For $k\geq 16$ define
	\[X_1 = \begin{bmatrix} 1 & 0_{1\times (k-1)}\\0_{15\times 1} & 0_{15\times (k-1)}\end{bmatrix},\;\;
		X_2 = \begin{bmatrix} I_{8\times 8} & 0_{8\times (k-8)}\\I_{8\times 8} & 0_{8\times (k-8)}\end{bmatrix},\;\;\textup{and}\;\; 
			X_3 = \begin{bmatrix} I_{16\times 16} & 0_{16\times (k-16)}\end{bmatrix}.\]
	We have that $q_{k}^\OO(X_1) = 1/4$, $q_k^{\OO}(X_2) = 64$ and $q_{k}^{\OO}(X_3) = 128$. 
	If $k=17$, by explicitly computing the coefficient matrix we see that $q_{17}^{\OO}$ is a sum of squares
	if and only if there exists $\lambda \in \RR_+^8$ such that 
	\[ \begin{bmatrix} \frac{1}{4}\\64\\128\end{bmatrix} =
 		\begin{bmatrix} \frac{1}{(16)(17)} &  \frac{1}{(16)(17)} & \frac{14}{(16)(17)} &  \frac{1}{17} &  
 			\frac{1}{17} & \frac{14}{17} & 0 & 0\\
 			\frac{16}{17} & \frac{16}{17} & 0 & \frac{18}{17} & \frac{18}{17} & 140 & 56 & 56\\
 		\frac{16}{17} & 0 & 0 & \frac{1}{17} & 9 & 126 & 36 & 84\end{bmatrix}\lambda\] 
		(where we have ordered the columns according to~\eqref{eq:lambda-order}).
				To see that this is infeasible, it is enough to multiply both sides on the left by
				$\begin{bmatrix} 252 & 3 & -2\end{bmatrix}$ to obtain
					\[ -1 = \begin{bmatrix} 127/68 & 15/4 & 441/34 & 304/17 & 0 & 6384/17 & 96 & 0\end{bmatrix}\lambda\]
						which clearly cannot be satisfied for any $\lambda\in \RR_+^8$.

	If $k=16$, by explicitly computing the coefficient matrix we see that $q_{16}^{\OO}$ is a sum of squares
	if and only if there exists $\lambda\in \RR_+^8$ such that 
	\[ \begin{bmatrix} \frac{1}{4}\\64\\128\end{bmatrix} = 
		\begin{bmatrix} \frac{1}{16^2} &  \frac{1}{16^2} & \frac{14}{16^2} &  \frac{15}{16^2} &  
 			\frac{15}{16^2} & \frac{210}{16^2} & 0 & 0\\
 			1 & 1 & 0 &  1& 1 & 140 & 56 & 56\\
 		1 & 0 & 0 & 0 & 9 & 126 & 36 & 84\end{bmatrix}\lambda.\]
	This is satisfied by choosing $\lambda = \begin{bmatrix} 0 &0 &0 &0 &64/15 & 0 &0 & 16/15 \end{bmatrix}$. 
		In other words, we have the sum of squares decomposition
		\[ q_{16}^{\OO}(X) = \frac{64}{15}s_{11}(X) + \frac{16}{15}s_{3,-1}(X).\]
\end{proof}

%% file: discussion.tex
\section{Discussion}
\label{sec:discussion}

We conclude with a discussion of some open questions related to this work.

\paragraph{Increasing the degree and number of variables}
Clearly, if $q$ is a convex form in $n$ variables of degree $2d$ that is not a sum of squares, then 
\[ \tilde{q}(x,x_0) = q(x) + x_0^{2d}\]
is convex and not a sum of squares. Therefore, we know that for all $n\geq 16\times 17$ there is an $n$-variate
quartic form that is convex but not a sum of squares. 

It would be very interesting to come up with a general construction of explicit
degree $2d$ convex forms that are not sums of squares for all $2d\geq 4$. 
However, it is unclear how to take an existing convex form that is not a sum of squares and increase its
degree while maintaining convexity and the property of not being a sum of
squares.  A similar difficulty was also faced by Ahmadi and Parrilo, when
constructing convex forms in all degrees that were not SOS-convex~\cite{ahmadi2013complete}. Instead,
they managed to come up with a method to directly produce a convex but not sos-convex
form of degree $2d+2$ in $n$ variables, given a non-negative form of degree
$2d$ in $n$ variables that is not a sum of squares. 

\begin{problem}
	Given a convex form that is not a sum of sqaures, find a way to construct a convex form of higher degree 
	that is also not a sum of squares.
\end{problem}
This would immediately yield examples of convex forms of degree $2d\geq 6$ that are not sums of squares.

\paragraph{Gap instances for polynomial optimization on the sphere}
While the results of this paper are phrased in terms of convex forms that are not sums of squares, 
our arguments are really focused on finding examples of quartic forms for which 
\[ \gap(p) = \frac{p_{\max} - p_{\min}^{\sos}}{p_{\max} - p_{\min}}>2.\]
In general, to find a degree $2d$ form that is convex but not a sum of squares,
it suffices to find a form $p$ of degree $2d$ with $\gap(p) > d$. 
There has been some study of upper bounds on this quantity~\cite{nie2012sum,fang2020sum},
(which grow like $n^{d/2}$ for fixed $d$), and it is known that polynomials with 
random coefficients give rise to large values of $\gap(\cdot)$~\cite{bhattiprolu2017sum} with high probability. 
However, very little appears to be known in terms of explicit gap instances. 
\begin{problem}
	Find an explicit family of forms $p_n$ of fixed degree $2d$ with increasing number of variables, 
	such that $\gap(p_n)$ grows without bound with $n$.
\end{problem}
Any such family would eventually give rise to further examples of convex forms that are
not sums of squares.

\paragraph{Explicit gaps for OT-FKM forms}
The Cauchy-Schwarz forms $\CS_k^\OO$ are closely related to a class of forms known as OT-FKM-type isoparametric forms, which 
are particular classes of solutions of the Cartan-M\"unzner equations~\cite{ozeki1975some,ferus1981cliffordalgebren}.
These are forms $F$ such that  
\[ F(x) = \|x\|^4 - 2\sum_{i=0}^{m}\langle P_i x,x\rangle^2\]
where the $2\ell\times 2\ell$ matrices $P_{i}$  form a Clifford system, 
i.e., they are symmetric matrices satisfying $P_iP_j + P_{j}P_i = 2\delta_{ij}I$. 
These forms take values between $-1$ and $1$ on the unit sphere, and so 
\[ p(x) = \frac{1}{2}(\|x\|^4+F(x)) = \|x\|^4 - \sum_{i=0}^{m}\langle P_{i}x,x\rangle^2\]
is nonnegative for all $x$. These vanish (on the sphere) 
on special isparametric submanifolds called \emph{focal submanifolds}. 
The Cauchy-Schwarz forms $\CS_{k}^{\OO}$ studied in this paper 
correspond (up to a scaling factor) to the case of $p$ where the Clifford system is given by the $16k\times 16k$ 
matrices $P_i = S_i\otimes I_k$ for $i=0,1,\ldots,8$.

A full classification of when the nonnegative forms $p$ (associated with OT-FKM-type isoparametric forms $F$)
are sums of squares is given in~\cite{ge2018isoparametric}. It would be interesting to make this classification
quantitative.
\begin{problem}
	Compute $p_{\min}^{\sos}$ for the nonnegative forms $p(x) =
	\frac{1}{2}(\|x\|^4+F(x))$ arising from OT-FKM-type isoparametric forms $F$. 
\end{problem}
By a variation on the proof of Theorem~\ref{thm:main-final},  
it is possible to show that $[\CS_{k}^{\OO}]^{\sos}_{\min} = -\frac{2(k-1)}{8+7k}$ for $k\geq 2$. 
As such, this family satisfies
$\gap(\CS_k^\OO) = \frac{15k}{8+7k}\rightarrow 15/7$ as $k\rightarrow \infty$. In particular, this 
does not grow without bound as the number of variables increases.

\paragraph{Examples in fewer variables}
Solving the following problem from~\cite{ahmadi2013complete} would complete our understanding of when convex quartics are sums of squares.
\begin{problem}
	Find the smallest $n$ such that there is a convex quartic form in $n$ variables that is not a sum of squares.
\end{problem}
We know that such a form must have $n\geq 5$ by El Khadir's work~\cite{el2020sum}. 
While it should be possible to find explicit examples of forms in fewer than $16\times 17$ variables
for which $\gap(p) > 2$, this is only a sufficient condition for a quartic form to be convex but not a 
sum of squares. Finding minimal examples will likely require working directly with convexity, rather than 
Blekherman's sufficient condition for convexity from Theorem~\ref{thm:blekherman}.

\paragraph{Invariant forms}
We have seen that convex forms are not always sums of squares when restricted to the 
three-dimensional subspace of $\Spin(9)\times O(k)$-invariant quartic forms in $16k$ variables. It would be 
interesting to investigate the relationship between convexity and sums of squares when 
restricted to $G$-invariant forms in $n$ variables for other group actions on $\RR^n$. A particularly canonical case
would be to study forms invariant under the action of the symmetric group by permuting the variables (symmetric forms).
While it is known that there are symmetric nonnegative forms that are not sums of squares whenever 
there are nonnegative
forms that are not sums of squares~\cite{goel2016choi}, it remains unclear whether the same is true for convex forms. 
\begin{problem}
	Determine whether all convex symmetric forms are sums of squares.
\end{problem}
More generally, suppose there were an (orthogonal) action of a group $G$ on $\RR^n$ such that all $G$-invariant convex forms of degree $2d$ are sums of squares. Then, because $\|x\|^{2d}$ is $G$-invariant, by essentially the same argument as given in Section~\ref{sec:worst-case}, it would follow that $\gap(p)<d$ for all elements of the subspace of 
$G$-invariant forms.